\documentclass[12pt,epsfig,amsfonts]{amsart} 
\usepackage{amsmath,amsthm,amssymb,amscd,epsfig,color}
\usepackage{graphicx}
\usepackage{mathrsfs}

\newtheorem*{theorema}{Theorem A (Large Deviation Principle)}
\newtheorem*{theoremb}{Theorem B (LDP for the three means)}
\newtheorem*{theoremc}{Theorem C (Rate functions with unbounded flat part)}

\newtheorem{prop}{Proposition}[section]
\newtheorem{lemma}[prop]{Lemma}

\theoremstyle{definition}

\theoremstyle{remark}
\newtheorem{remark}[prop]{Remark}

\numberwithin{equation}{section}

\begin{document}

\author{Hiroki Takahasi}

\address{Keio Institute of Pure and Applied Sciences (KiPAS), Department of Mathematics,
Keio University, Yokohama,
223-8522, JAPAN} 
\email{hiroki@math.keio.ac.jp}
\urladdr{\texttt{http://www.math.keio.ac.jp/~hiroki/}}



\title[LDP in the BCF expansion]
 {Large deviation principle for\\ harmonic/geometric/arithmetic mean of digits\\ in
 backward continued fraction expansion} 
 \maketitle
 
 \begin{abstract}
 We establish the (level-1) large deviation principles for three kinds of means
 associated with the backward continued fraction expansion.
  We show that: for the harmonic and geometric means, 
   the rate functions vanish exactly at one point;
  for the arithmetic mean, it is completely degenerate, vanishing at every point
  in its effective domain. Our method of proof employs the thermodynamic formalism for finite Markov shifts,
 and a multifractal analysis for the R\'enyi map generating the backward continued fraction digits.
 We completely determine the class of unbounded arithmetic functions for which
 the rate functions vanish at every point in unbounded intervals.
  \end{abstract}
  
  \section{Introduction}
  Each number $x\in(0,1)\setminus\mathbb Q$ 
  has 
  {\it the Regular Continued Fraction} (RCF) expansion
 \begin{equation}\label{forward}x=\cfrac{1}{a_1+\cfrac{1}{a_2+\cfrac{1}{\ddots}}},\end{equation}
 where each $a_j=a_j(x)$ is a positive integer.
Statistics of the digits $a_1(x),\ldots,a_n(x)$ as $n\to\infty$ for typical $x$ in the sense of the Lebesgue measure
$\lambda$ on $[0,1]$
has been considered since the time of Carl Friedrich Gauss.
Asymptotics of the following three means are of particular interest:
\medskip

(a) (harmonic mean) $n(a_1^{-1}+\cdots+a_n^{-1})^{-1}$;

(b) (geometric mean) $\sqrt[n]{a_1\cdots a_n}$;

(c) (arithmetic mean) $n^{-1}(a_1+\cdots+a_n)$.
\medskip

\noindent Among these is a well-known relation
$$\frac{n}{a_1^{-1}+\cdots+a_n^{-1}}\leq 
\sqrt[n]{a_1\cdots a_n}\leq \frac{a_1+\cdots+a_n}{n}.$$

It was Khinchin \cite[Theorem 35]{Khi64} who showed that the geometric mean converges to a constant
$K=\prod_{k=1}^\infty\left(1+\frac{1}{k(k+2)}\right)^{\frac{\log k}{\log2}}$  $\lambda$-a.e.
Although Khinchin did not use ergodic theory in his original proof, his result is a consequence
of Birkhoff's ergodic theorem applied to the Gauss map 
 $x\in(0,1]\to 1/x-\lfloor1/x\rfloor\in[0,1)$ and the Gauss measure $\frac{1}{\log2}\frac{dx}{1+x}$.
 By Birkhoff's ergodic theorem, the harmonic mean converges to a constant
$H=\left(\frac{1}{\log2}\sum_{k=1}^\infty \frac{1}{k}\log\frac{(k+1)^2}{k(k+2)}\right)^{-1}$  $\lambda$-a.e.
The arithmetic mean divergences to infinity $\lambda$-a.e. \cite{Khi64}.

There are a number of ways to expand real numbers into continued fractions
(See e.g., \cite{DajKra00,IosKra02}).
We consider
 {\it the Backward Continued Fraction} (BCF) expansion
    \begin{equation*}\label{expansion}x=1-\cfrac{1}{b_1-\cfrac{1}{b_2-\cfrac{1}{\ddots}}},\end{equation*}
    where  $x\in[0,1)$ and each $b_j=b_j(x)$ is an integer greater than or equal to $2$.
          The BCF expansion has been used, for example, in inhomogeneous Diophantine approximations
    \cite{Pin01}.
 For connections of the BCF expansion with geodesic flows,
    see \cite{AdlFla84}.
     The digits in the BCF expansion are generated by iterating the R\'enyi map
    $T\colon[0,1)\to[0,1)$ given by
     $$Tx=\frac{1}{1-x}-\left\lfloor\frac{1}{1-x}\right\rfloor,$$ namely
     $$b_n(x)=\left\lfloor\frac{1}{1-T^{n-1}(x)}\right\rfloor+1\quad\text{for each $n\geq1$.}$$
         The graph of $T$ can be obtained from that of the Gauss map 
      by reflecting the latter in the line $x=1/2$. This is the reason for the name of the
      BCF expansion.
      
       The map $T$
    is a full-branched 
     non-uniformly expanding Markov map having 
    $x=0$ as a unique neutral fixed point.
    R\'enyi \cite{Ren57} showed that $T$ leaves invariant the measure $\frac{dx}{x}$.
    This is an ergodic yet infinite measure, and so Birkhoff's ergodic theorem does not apply. 
     As a result, the statistics of the digits in the BCF expansion 
     is very much different from that in the RCF expansion.
    Using
      Hopf's ratio ergodic theorem 
      one can show the existence of intermittency at the neutral fixed point,
       in particular 
      $$\lim_{n\to\infty}\frac{1}{n}\#\{1\leq j\leq n\colon b_j=2\}=1\quad \text{$\lambda$-a.e.}$$ 
    Iosifesqu and Kraaikamp \cite[Theorem 4.4.8]{IosKra02} showed that the harmonic and geometric means of the 
BCF digits converge
to $2$ $\lambda$-a.e. On the other hand,
typical behaviors of the arithmetic mean 
are peculiar. 
Aaronson \cite{Aar86} showed that the arithmetic mean converges to $3$ in measure.
Aaronson and Nakada \cite{AarNak03} proved that
$$\liminf_{n\to\infty}\frac{b_1+\cdots +b_n}{n}=2\quad\text{and }\quad\limsup_{n\to\infty}\frac{b_1+\cdots +b_n}{n}=\infty
\quad \text{$\lambda$-a.e.}$$

In this paper we are concerned with the large deviation principle (LDP) for the BCF expansion, 
namely, with the existence of rate functions
which estimate probabilities of rare events with which the harmonic/geometric/arithmetic means stay away from their expectations.
Although the arithmetic mean has no expectation due to Aaronson and Nakada, it is still meaningful to consider the existence of a rate function.
We emphasize that the formulation of the LDP due to Donsker and Varadhan 
\cite{DonVar75} does not a-priori require the existence of expectation.

For a comparison later, let us summarize known results from \cite{DenKab07,Tak19} on the LDP for the RCF expansion.

\begin{itemize}
\item[(a)] (harmonic mean \cite{DenKab07})
There exists a strictly convex function $I^{({\rm a})}\colon\mathbb R\to[0,\infty]$ such that the following holds:

\begin{itemize}

\item for any $\alpha\in[1,H)$,
\begin{align*}\lim_{n\to\infty}\frac{1}{n}\log\lambda\left\{x\in(0,1)\setminus\mathbb Q\colon
\frac{n}{a_1^{-1}(x)+\cdots +a_n^{-1}(x)}\leq \alpha \right\}
&=- I^{({\rm a})}\left(\frac{1}{\alpha}\right)\\
&<0.\end{align*}

\item for any $\alpha\in(H,\infty)$,
\begin{align*}\lim_{n\to\infty}\frac{1}{n}\log\lambda\left\{x\in(0,1)\setminus\mathbb Q\colon
\frac{n}{a_1^{-1}(x)+\cdots +a_n^{-1}(x)}\geq \alpha \right\}
&=- I^{({\rm a})}\left(\frac{1}{\alpha}\right)\\
&<0.\end{align*}
\end{itemize}

\item[(b)] (geometric mean \cite[Theorem 2]{Tak19})
There exists a convex function $I^{({\rm b})}\colon\mathbb R\to[0,\infty]$ such that for any interval $J$
intersecting $(2,\infty)$,
\begin{align*}\lim_{n\to\infty}\frac{1}{n}\log\lambda\left\{x\in(0,1)\setminus\mathbb Q\colon
\sqrt[n]{a_1(x)\cdots a_n(x)}\in J\right\}
&=- \inf_JI^{({\rm b})}(\log\alpha).\end{align*}
It is not precisely known for which $J$ this rate is strictly negative.
 It is plausible that $I^{({\rm b})}$ vanishes only at $\log K$,
 and so this rate is strictly negative if and only if $\log K\notin{\rm cl}(J)$.

\item[(c)] (arithmetic mean \cite[Theorem 2]{Tak19}) there exists a strictly positive, strictly decreasing convex function 
$I^{({\rm c})}\colon\mathbb R\to[0,\infty]$ such that
$\displaystyle{\lim_{\alpha\to\infty}}I^{({\rm c})}(\alpha)=0$ and
\begin{align*}
\lim_{n\to\infty}\frac{1}{n}\log\lambda\left\{x\in(0,1)\setminus\mathbb Q\colon
\frac{a_1(x)+\cdots +a_n(x)}{n}\leq \alpha \right\}&=-I^{({\rm c})}(\alpha)\\
&<0
\end{align*}
for any $\alpha\geq1$. Moreover,
 for any $\alpha\geq 1$,
\begin{align*}\lim_{n\to\infty}\frac{1}{n}\log\lambda\left\{x\in(0,1)\setminus\mathbb Q\colon
\frac{a_1(x)+\cdots +a_n(x)}{n}\geq \alpha \right\}=0.
\end{align*}

\end{itemize}

The approach in \cite{Tak19} differs from that in \cite{DenKab07} and 
 uses finite approximations based on ergodic theory and thermodynamic formalism inspired by \cite{Tak84}.
This approach is more flexible than that in \cite{DenKab07} and allows cases in which there is no expectation
and the strong law of large number does not hold. Moreover, it gives a representation of the rate function
in terms of entropy and Lyapunov exponent of invariant measures. 
We pursue the approach in \cite{Tak19}
for the BCF expansion and the R\'enyi map.

Let $\mathcal M(T)$ denote the set of $T$-invariant Borel probability measures.
 Write $\phi=-\log T'$ and denote by
$\mathcal M_\phi(T)$ the set of elements of $\mathcal M(T)$ for which $\phi$ is integrable.
For each $\mu\in\mathcal M_\phi(T)$ define
$\chi(\mu)=-\int\phi d\mu\geq0,$ and let $h(\mu)$ denote the 
Kolmogorov-Sina{\u\i} entropy of $\mu$ with respect to $T$.
 Put $$F(\mu)=h(\mu)-\chi(\mu).$$

Let $\mathbb N$ denote the set of non-negative integers.
By {\it an arithmetic function} we mean 
a non-constant function $\mathbb N\setminus\{0,1\}\to\mathbb R$.
We view $\psi\circ b_1$ as an observable and consider its time averages along orbits of $T$.
For an arithmetic function $\psi$ define
 a rate function $I_{\psi\circ b_1}\colon\mathbb R\to[0,\infty]$ by
$$I_{\psi\circ b_1}(\alpha)=-\lim_{\epsilon\to0}\sup\left\{F(\mu)\colon\mu\in\mathcal M_\phi(T),
\int \psi\circ b_1d\mu\in(\alpha-\epsilon,\alpha+\epsilon)\right\},$$
where $\sup\emptyset=-\infty$. The supremum is taken over all 
$\mu\in\mathcal M_\phi(T)$ for which $\psi\circ b_1$ is integrable.
By definition, this is a lower semi-continuous function.
Define
$$c_\psi=\inf_{x\in[0,1)}\liminf_{n\to\infty}\frac{1}{n}\sum_{j=1}^n\psi(b_j(x)),$$
and
$$d_\psi=\sup_{x\in[0,1)}\limsup_{n\to\infty}\frac{1}{n}\sum_{j=1}^n\psi(b_j(x)).$$
Then we have $-\infty\leq c_\psi< d_\psi\leq\infty$.
It is not hard to show that
these two numbers determine an effective domain on which the rate function is bounded:
if $\alpha\in(c_\psi,d_\psi)$ then
$I_{\psi\circ b_1}(\alpha)<\infty$;
if $\alpha\notin[c_\psi,d_\psi]$ then
$I_{\psi\circ b_1}(\alpha)=\infty$.
From the affinity of entropy and Lyapunov exponent on measures,
the rate function is a convex function.
\begin{theorema}
Let $\psi$ be an arithmetic function.
For any interval $J$ intersecting $(c_\psi,d_\psi)$,
$$\lim_{n\to\infty}\frac{1}{n}\log\lambda\left\{x\in[0,1)\colon\frac{1}{n}
\sum_{j=1}^n\psi(b_j(x))\in J\right\}
=-\inf_J I_{\psi\circ b_1}.$$
\end{theorema}

For the RCF expansion, the LDP in the form of Theorem A was shown in \cite{DenKab07} for bounded
$\psi$, and in \cite{Tak19} in full generality. All new difficulties here arise from the existence of the neutral fixed point
of the R\'enyi map $x=0$.
 For interval maps with neutral fixed points like the Manneville-Pomeau map
 $x\in[0,1]\mapsto x+x^{1+p}\in[0,1]$ mod $\mathbb Z$ $(p>0)$,
 several large deviations results for continuous observables are available
  \cite{PolSha09,PolShaYur}.
The R\'enyi map has an infinite number of branches, and so does not fall into this class.
Moreover, the function $\psi$ in Theorem A is not assumed to be bounded.
The unboundedness of observables brings new difficulties and new phenomena,
as developed below.

Theorem A is so general, being true for any arithmetic function, that nothing further can be said about
 properties of rate functions in this full generality. 
In order to conclude for which $J$ this rate is strictly negative in each specific case,
it is necessary to identify the set of minimizers.
For the above three specific means in question, we obtain the next result.

\begin{theoremb}
The following holds:
\begin{itemize}
\item[(a)] {\rm (harmonic mean)} 
 $I_{b_1^{-1}}(\alpha)=0$ if and only if $\alpha=1/2$.
In particular, for any $\alpha\in[2,\infty)$,
\begin{align*}\lim_{n\to\infty}\frac{1}{n}\log\lambda\left\{x\in[0,1)\colon
\frac{n}{b_1^{-1}(x)+\cdots +b_n^{-1}(x)}\leq \alpha \right\}
&=- I_{b_1^{-1}}\left(\frac{1}{\alpha}\right)\\
&<0.\end{align*}
\item[(b)] {\rm (geometric mean)}
$I_{\log b_1}(\alpha)=0$ if and only if $\alpha=\log2$. In particular,
for any $\alpha\in[2,\infty)$,
\begin{align*}\lim_{n\to\infty}\frac{1}{n}\log\lambda\left\{x\in[0,1)\colon
\sqrt[n]{b_1(x)\cdots b_n(x)}\geq \alpha \right\}
&=- I_{\log b_1}(\log\alpha)\\&<0.\end{align*}
\item[(c)] {\rm (arithmetic mean)}
 $I_{b_1}(\alpha)=0$ for any $\alpha\in[2,\infty)$. 
For every $\delta>5$ and
  every bounded interval $J\subset [2,\infty)$, there exists $N=N(\delta,J)\geq1$ such that for every $n\geq N$,
 \begin{align*}\lambda\left\{x\in[0,1)\colon
\frac{b_1(x)+\cdots +b_n(x)}{n}\in J \right\}\geq\frac{n^{-\delta}}{(\sup J-2)^2}.\end{align*}
\end{itemize}
\end{theoremb}

Given an arithmetic function $\psi$,
let us call
$\alpha\in\mathbb R$ {\it a minimizer} 
if $I_{\psi\circ b_1}(\alpha)=0$ holds.
Theorem A states that the minimizer is unique for the harmonic and geometric means, while
the whole effective domain $[2,\infty)$ coincides with the set of minimizers for the arithmetic mean, 
as a consequence of the sub-exponential lower large deviations bound.

An identification of the set of minimizers is a non-trivial problem.
Since the unit point mass $\delta_0$ at the neutral fixed point satisfies $F(\delta_0)=0$,
 $\psi(2)$ is a minimizer for any $\psi$.
 It is true that $F(\mu)=0$ if and only if $\mu=\delta_0$.
 However, there can exist minimizers other than $\psi(2)$.
Note that $\alpha\in\mathbb R$ is a minimizer
if and only if there is a sequence $\{\mu_k\}_{k=1}^\infty$ in $\mathcal M_\phi(T)$
such that $$\lim_{k\to\infty}F(\mu_k)=0\ \text{ and }\ \lim_{k\to\infty}\int\psi\circ b_1d\mu_k=\alpha.$$
This does not imply $\alpha=\psi(2)$, primarily because
$\mu\in\mathcal M_\phi(T)\mapsto F(\mu)$ is not continuous. 
Even the tightness of $\{\mu_k\}_{k=1}^\infty$ does not follow in the case
$\psi$ is bounded like the harmonic mean case.
To identify the set of minimizers, our
idea is to look at entropy divided by Lyapunov exponent rather than their difference, namely
the dimension of a measure, and use results on the multifractal analysis of the R\'enyi map
\cite{Iom10,JaeTak}.



It is still not well-understood for which dynamical systems and observables
the uniqueness of minimizer breaks down.
If observables are bounded continuous,
the size of the set of minimizers 
 seem to have a correlation with the strength of hyperbolicity of the system.
For uniformly hyperbolic systems, the minimizer is always unique \cite{Kif90,OrePel89,Tak84}.
For non-uniformly hyperbolic systems admitting inducing schemes with exponential tail, 
there exists a rate function locally defined around the mean value
which vanishes only at the mean \cite{MelNic08,ReyYou08}.
The Manneville-Pomeau map with $p<1$ is an archetypal model in which the uniqueness of minimizer breaks down
 \cite{PolSha09,PolShaYur}.
This map has sub-exponential decay of correlations \cite{Gou04,Hu04,Sar02}, and 
 this rate is intimately related to sub-exponential large deviation rates \cite{MelNic08}.
 For unimodal maps satisfying the Collet-Eckmann condition \cite{ColEck83}, the minimizer is unique \cite{ChuTak,KelNow}.
 For those with very weak hyperbolicity constructed by Hofbauer and Keller \cite{HofKel90},
the rate functions vanish on substantially large intervals \cite[Appendix]{ChuRivTak19}. 
Large deviation results for unbounded observables are much more limited
(see e.g., \cite{NicTor19,Tak19}), primarily because
the usual transfer operator methods do not work.

For the R\'enyi map, the neutral fixed point assists in the identification of the sets of minimizers.
In the case of unbounded observables,
a comparison of items (b), (c) in Theorem B suggests that
the speed of blow-up of an observable at $x=1$
is an important factor on the size of the set of minimizers.
Indeed,
in the proof of (c),
we explicitly construct sub-exponentially large sets
using orbits from a neighborhood of $x=0$ to that of $x=1$.
A key point is that the first digit function $ b_1(x)$ blows up faster
than $\log T'(x)$ does as $x\to1$. In this construction,
the constant $\delta$ cannot be replaced by any number smaller than $5=1+2+2$.
These three numbers originate from:
Thaler's asymptotic formula \cite{Tha83} at $x=0$; the quadratic 
order of sizes of elements of the Markov partition near $x=1$;
R\'enyi's condition on distortions.

The above explicit construction for the arithmetic mean can be extended to some other arithmetic functions with
sufficiently fast speeds of blow-up at $x=1$.
Nevertheless, a close inspection into the construction shows that interesting arithmetic functions like
$$\psi(n)=\begin{cases}n&\text{ if $n$ is prime}\\
0&\text{ otherwise}\end{cases}$$
cannot be covered. Then it is natural to ask for which unbounded $\psi$ the corresponding rate function vanishes 
on the whole unbounded interval. The next theorem completely determines the class of such $\psi$.

\begin{theoremc}\label{heavy}
Let $\psi$ be an arithmetic function such that $\displaystyle{\liminf_{n\to\infty}}$$\psi(n)\geq0$ and 
$\displaystyle{\limsup_{n\to\infty}}$$\psi(n)=\infty.$
Then  $I_{\psi\circ b_1}(\alpha)=0$ 
  holds for any $\alpha\in[\psi(2),\infty)$ if and only if
 $$\limsup_{n\to\infty} \frac{\psi(n)}{\log n}=\infty.$$ \end{theoremc}
 
 To prove Theorem C, we replace the construction of sub-exponentially large sets
 with a multifractal-analysis argument which shows that the rate function vanishes at infinity.

\section{Proofs of the theorems}
In this section we prove the three theorems (Theorem A in Sect.\ref{LDP},
Theorem B in Sect.\ref{identifying}, Theorem C in Sect.\ref{flat-part}).
Sect.\ref{MFA} collects ingredients in the multifractal analysis needed for the proofs of Theorem B and Theorem C.


\subsection{The Large deviation principle}\label{LDP}

\begin{proof}[Proof of Theorem A]
Let $\psi$ be an arithmetic function.
Since the rate function $I_{\psi\circ b_1}$ is convex, it is continuous on $(c_\psi,d_\psi)$.
If $c_\psi$ is finite and $I_{\psi\circ b_1}(c_\psi)<\infty$, then from the lower semi-continuity and convexity
of the rate function,
$\displaystyle{\lim_{\alpha\to c_\psi+0}}$$I_{\psi\circ b_1}(\alpha)=I_{\psi\circ b_1}(c_\psi)$ holds.
The continuity also holds at the other boundary point $d_\psi$ provided it is finite.
Therefore, it suffices to prove the following lower and upper bounds:
\begin{itemize}
\item[] for any open interval $U$ intersecting $(c_\psi,d_\psi)$,
\begin{equation}\label{l}
\liminf_{n\to\infty}\frac{1}{n}\log\lambda\left\{x\in[0,1)\colon\frac{1}{n}
\sum_{j=1}^n\psi(b_j(x))\in U\right\}
\geq-\inf_U I_{\psi\circ b_1}.\end{equation}
\item[] for any closed interval $V$ intersecting $(c_\psi,d_\psi)$,
\begin{equation}\label{u}\limsup_{n\to\infty}\frac{1}{n}\log\lambda\left\{x\in[0,1)\colon\frac{1}{n}
\sum_{j=1}^n\psi(b_j(x))\in V\right\}
\leq-\inf_V I_{\psi\circ b_1}.\end{equation}
\end{itemize}
We match these bounds to deduce the desired equality.

 Although proofs of the two bounds
 proceed much in parallel to that of \cite[Theorem 2]{Tak19} for the RCF expansion,
some additional considerations are necessary due to the existence of the neutral fixed point
(e.g., a sub-exponential distortion).
A proof of the next lemma is completely analogous to that of \cite[Lemma 8]{Tak19} and hence omitted.
\begin{lemma}\label{ratefunction}
For any Borel set $B\subset\mathbb R$,
$$\sup \left\{F(\mu)\colon\mu\in\mathcal M_\phi(T),\ \int\psi\circ b_1d\mu\in B\right\}\leq-\inf_B I_{\psi\circ b_1}.$$
If $B$ is an open set, then the inequality is an equality.
\end{lemma}

The lower bound \eqref{l} is a consequence of Lemma \ref{ratefunction} and the next lemma.

\begin{lemma}\label{lowlem}
Let $U$ be an open interval intersecting $(c_\psi,d_\psi)$.
For any measure $\mu\in\mathcal M_\phi(T)$ satisfying $\psi\circ b_1\in L^1(\mu)$ and $\int\psi\circ b_1d\mu\in U$,
$$\liminf_{n\to\infty}\frac{1}{n}\log\lambda\left\{x\in[0,1)\colon
\frac{1}{n}\sum_{j=1}^n\psi(b_j(x))\in U\right\}\geq F(\mu).$$
\end{lemma}
\begin{proof}
If $\mu$ is non-ergodic, then for any $\epsilon>0$
there exists an ergodic $\mu_\epsilon\in\mathcal M_\phi(T)$ such that
$|\int\psi\circ b_1d \mu-\int\psi\circ b_1d \mu_\epsilon|<\epsilon$ and $|F(\mu)-F(\mu_\epsilon)|<\epsilon$
\cite{IomJor,JaeTak}.
Hence, it suffices to show the desired inequality assuming $\mu$ is ergodic.

Let $\epsilon>0$ be such that $(\int\psi\circ b_1d\mu-\epsilon,\int\psi\circ b_1d\mu+\epsilon)\subset U$.
For each integer $n\geq1$
let $\mathscr{A}^n$ denote the collection of maximal intervals on which $T^n$ is well-defined
and continuous. 
For each $j=1,\ldots,n$, 
$b_j$ is constant on each element $A\in\mathscr{A}^n$.
This constant value is denoted by $b_j(A)$.
Since  $\mu$ is ergodic and $\mathscr{A}^1$ is a generator,
by Birkhoff's ergodic theorem and Shannon-McMillan-Breiman's theorem,
there exists $n_0\geq 1$ such that for each integer $n\geq n_0$
there exists a finite subset $\mathscr B$ of $\mathscr{A}^n$ such that
$$\left|\frac{1}{n}\log\# \mathscr B-h(\mu)\right|<\epsilon,$$  
$$\left|\frac{1}{n}\sum_{j=1}^n\psi(b_j(A))-\int\psi\circ b_1d\mu\right|<\epsilon\quad\text{and}\quad
\sup_A\left|\frac{1}{n}\log (T^n)'-\chi(\mu)\right|<\epsilon$$
for each $A\in \mathscr B$. Therefore
\begin{align*}
\frac{1}{n}\log\lambda\left\{x\in[0,1)\colon\frac{1}{n}\sum_{j=1}^n\psi(b_j(x))\in U\right\}
&\geq\sum_{A\in\mathscr{B}}\lambda(A)\\
&\geq F(\mu)-2\epsilon.\end{align*}
Increasing $n$ to $\infty$ and decreasing $\epsilon$ to $0$ completes the proof.
\end{proof}

To prove the upper bound \eqref{u},
let $V$ be a closed interval intersecting $(c_\psi,d_\psi)$.
For an integer
  $n\geq1$ define
  $$ \mathscr{A}^n(V)=
\left\{A\in\mathscr{A}^n\colon  \frac{1}{n}\sum_{j=1}^n\psi(b_j(A))\in V\right\}.$$
This is a non-empty set for sufficiently large $n$.
 \begin{lemma}\label{horse}
 For any $\gamma>2$ there exists $N\geq1$ such that if $n\geq N$ then
 for each non-empty finite subset $\mathscr{C}$ of $\mathscr{A}^n(V)$, 
there exists a measure $\mu\in\mathcal M_\phi(T)$ such that 
$$  \log\lambda[\mathscr{C}]\leq F(\mu)n+\gamma\log n \ \ \text{ and }
 \ \   \int\psi\circ b_1d\mu\in V,$$
where $[\mathscr{C}]$ denotes the union of all elements of $\mathscr{C}$.
 \end{lemma}
 \begin{proof}
 Let $\gamma>2$.
  Put $\widehat T=T^n$ and 
 $\Lambda=\bigcap_{k=0}^\infty{\widehat T}^{-k}[\mathscr{C}]$.
Then $\widehat T|_\Lambda\colon \Lambda\to \Lambda$ is topologically 
 conjugate to the one-sided full shift on $\#\mathscr{C}$-symbols.
 The function $\widehat\phi=-\log(\widehat T)'$ 
 induces a continuous potential on the shift space. The distortion is sub-exponential,
 and hence inconsequential in the LDP in which exponential scales are concerned.
 \begin{lemma}\label{distortion}
There exists $N\geq1$ such that for every $n\geq 1$ and every $A\in\mathscr{A}^n$,
$$\sup_{x,y\in A}\widehat\phi(x)-\widehat\phi(y)\leq  \gamma\log n.$$
\end{lemma}
\begin{proof}
Define a strictly decreasing sequence $\{c_n\}_{n=0}^{\infty}$ in $[0,1/2]$
inductively by $c_0=1/2$ and $Tc_n=c_{n-1}$ for $n\geq1$.
Since $\sup_{[0,1)} |T''|/|T'|^2\leq 2$,
 considering
the inverse branches of $T$, 
for all $A\in\mathscr{A}^1$ and $x,y\in A$ we have
$$\phi(x)-\phi(y)\leq\sup_{[0,1)} \frac{|T''|}{|T'|^2}|Tx-Ty|\leq2|Tx-Ty|.$$
Iterating this argument, for every $n\geq1$, $A\in\mathscr{A}^n$ and all $x,y\in A$,
\begin{equation*}\widehat\phi(x)-\widehat\phi(y)\leq 2\sum_{j=1}^n\lambda(T^jA)\leq2\sum_{j=1}^{n}\lambda(T^j[0,c_{n-1}])
\leq2\sum_{j=1}^{n-1} c_{n-j-1}+2.\end{equation*}
The second inequality is due to the fact that the longest element of $\mathscr{A}^{n-j}$
is the one containing $0$, which can be checked by induction.
Since $c_nn\to 1$ as $n\to\infty$ 
 (see \cite[Lemma 2, Corollary]{Tha83}),
 the desired estimate holds.
\end{proof}

 Lemma \ref{distortion} gives
 $$\lim_{m\to\infty}\frac{1}{m}\sup_{[c_0,\ldots,c_{m-1}]}\sup_{x,y\in[c_0,\ldots,c_{m-1}]}\sum_{k=0}^{m-1}(\widehat
 \phi(\widehat T^kx)-\widehat
 \phi(\widehat T^ky))=\lim_{m\to\infty}\frac{\gamma\log(mn)}{m}=0,$$ 
where $[c_0,\ldots,c_{m-1}]$ denotes the $m$-cylinder in the shift space
with a symbol sequence $c_0\cdots c_{m-1}$.
This property allows us to slightly modify
 the proof of the variational principle \cite[Lemma 1.20]{Bow75} 
 to show that
 \begin{equation}\label{equation}
 \begin{split}\sup_{\widehat\nu\in\mathcal M(\widehat T|_\Lambda)}\left(h_{\widehat T|_\Lambda}(\widehat\nu)+\int\widehat\phi d\widehat\nu\right)&=
 \lim_{m\to\infty}\frac{1}{m}\log\left(\sum_{x\in (\widehat T|_\Lambda)^{-m}(y_0)}
 \exp{\sum_{k=0}^{m-1}\widehat
 \phi(\widehat T^kx)}\right).\end{split}\end{equation}
   Here, $y_0\in \Lambda$ is fixed, $\mathcal M(\widehat T|_\Lambda)$ denotes 
 the space of $\widehat T|_\Lambda$-invariant Borel probability measures
 endowed with the weak*-topology,
 and $h_{\widehat T|_\Lambda}(\widehat\nu)$ denotes 
 the entropy of $\widehat\nu\in \mathcal M(\widehat T|_\Lambda)$
 with respect to $\widehat T|_\Lambda$.
For the summand inside the logarithm, we have
 \begin{equation}
 \begin{split}\label{sarf}\sum_{x\in (\widehat T|_\Lambda)^{-m}(y_0)}\exp\left(\sum_{k=0}^{m-1}\widehat\phi(\widehat T^kx)\right)
&\geq\left(\inf_{y'\in \Lambda}\sum_{x\in (\widehat T|_\Lambda)^{-1}(y')}e^{\widehat\phi(x)}\right)^m\\
&\geq\left(\sum_{A\in\mathscr{C}}\inf_{A}e^{\widehat\phi}\right)^m\\
&\geq\left(n^{-\gamma}\lambda[\mathscr{C}]\right)^m.
\end{split}\end{equation}
The last inequality is by Lemma \ref{distortion}.
Taking logarithms of both sides, dividing by $m$ increasing $m$ to $\infty$ gives
$$\lim_{m\to\infty}\frac{1}{m}\log\left(\sum_{x\in (\widehat T|_\Lambda)^{-m}(y_0)}
\exp\sum_{k=0}^{m-1}\widehat\phi(\widehat T^kx)\right)
\geq\log\lambda[\mathscr{C}]-\gamma\log n.$$
Plugging this into the previous inequality yields
$$\sup_{\widehat\nu\in\mathcal M(\widehat T|_\Lambda)}\left(h_{\widehat T|_\Lambda}(\widehat\nu)+\int\widehat\phi d\widehat\nu\right)\geq
\log\lambda[\mathscr{C}]-\gamma\log n.$$
Since $\mathcal M(\widehat T|_\Lambda)$ is compact 
and the map $\mathcal M(\widehat T|_\Lambda)\ni\widehat\nu\mapsto h_{\widehat T|_\Lambda}(\widehat\nu)+\int\widehat\phi d\widehat\nu $
is upper semi-continuous, there exists a measure $\widehat\mu$ in $\mathcal M(\widehat T|_\Lambda)$
which attains the above supremum. The spread measure
$\mu = (1/n)\sum_{i=0}^{n-1}\widehat\mu\circ T^{-i}$ is in $\mathcal M_\phi(T)$
and satisfies $\int\psi\circ b_1d\mu\in V$.
 \end{proof}

To finish, note that $\mathscr{A}^n(V)$ can be an infinite set, for instance in the case $\psi$ is bounded.
For each $\kappa\in(1,2)$
choose a finite subset $\mathscr{C}_\kappa$ of $\mathscr{A}^n(V)$ such that
$\lambda[\mathscr{A}^n(V)]\leq \kappa\lambda[\mathscr{C}_\kappa].$
By Lemma \ref{horse}, there exists $\mu\in\mathcal M_\phi(T)$ such that
$\log\lambda[\mathscr{C}_\kappa]\leq F(\mu)n+\gamma\log n$ and
$\int\psi\circ b_1 d\mu\in V.$
We have
\begin{align*}
\log\lambda\left\{x\in[0,1)\colon\frac{1}{n}\sum_{j=1}^n\psi(b_j(x))\in V\right\}&=
\log\lambda[\mathscr{A}^n(V)]\\
&\leq  \log \kappa+\log\lambda[\mathscr{C}_\kappa]\\
&\leq \log\kappa +F(\mu)n+ \gamma\log n\\
&\leq \log\kappa -n\inf_{V} I_{\psi\circ b_1}+\gamma\log n.
\end{align*}
The last equality is by Lemma \ref{ratefunction}.
Decreasing $\kappa$ to $1$, and then
dividing both sides by $n$, increasing $n$ to $\infty$ yields
\eqref{u}. This completes the proof of Theorem A.
\end{proof}


\subsection{Multifractal analysis for the R\'enyi map}\label{MFA}
      For $\alpha\in[0,\infty)$ define 
$$L(\alpha)=\left\{x\in [0,1)\colon \lim_{n\to\infty}\frac{1}{n}\log(T^n)'(x)=\alpha\right\}.$$
Define {\it a Lyapunov spectrum} $\alpha\in[0,\infty)\mapsto \mathcal L(\alpha)$ by
$$\mathcal L(\alpha)=\dim_HL(\alpha),$$
where $\dim_H$ denotes the Hausdorff dimension on $\mathbb R$ with respect to the Euclidean metric.
Iommi \cite[Theorem 4.2]{Iom10} proved that
$L(\alpha)$ is a non-empty set for every $\alpha\in[0,\infty)$, and $\alpha\in (0,\infty)\mapsto \mathcal L(\alpha)$
is strictly monotone decreasing. He gave a formula
  for $\mathcal L(\alpha)$ in terms of the Legendre transform 
of the geometric pressure function. For the proof of Theorem B we need a
conditional variational formula for $\mathcal L(\alpha)$
using the dimensions of expanding measures.

A measure $\mu\in\mathcal M_\phi(T)$ is called {\it expanding} if $\chi(\mu)>0$.
 The unit point mass at the neutral fixed point $x=0$
is the only one measure which is not expanding.
For an expanding measure $\mu$ define its {\it dimension}
$$\dim(\mu)=\frac{h(\mu)}{\chi(\mu)}.$$

\begin{lemma}\label{lalpha}\cite[Lemma 4.2]{JaeTak}
  There exists an ergodic expanding measure with dimension arbitrarily close to $1$.
  Moreover,
 for every $\alpha\in(0,\infty)$,
   \begin{equation*}
   \mathcal L(\alpha)=\max\left\{
  \lim_{\epsilon\to0}\sup\left\{\dim(\mu)\colon \mu\in\mathcal M_\phi(T),\ {\rm expanding,}\ \left|\chi(\mu)-\alpha\right|<\epsilon\right\},\frac{1}{2}\right\}.\end{equation*}
\end{lemma}

\begin{remark}In fact, $\mathcal L(\alpha)>1/2$ holds for every $\alpha>0$ but we will not need this.
\end{remark}

The uniqueness of minimizer of the rate functions for the harmonic and geometric means rests on the following lemma.
\begin{lemma}\label{upbound}
Let $\{\mu_k\}_{k=1}^\infty$ be a sequence of expanding measures such that
$\inf_{k\geq1}\chi(\mu_k)>0$. Then $$\sup_{k\geq1} \dim(\mu_k)<1.$$
\end{lemma}
\begin{proof}
We have
$$\sup_{k\geq1} \dim(\mu_k)\leq\sup_{k\geq1} \mathcal L(\chi(\mu_k))\leq\mathcal L\left(\inf_{k\geq1}\chi(\mu_k)\right)<1.$$
The first inequality is a consequence of the conditional variational formula in Lemma \ref{lalpha}. The second and third ones
are consequences of the strict monotonicity of the Lyapunov spectrum in \cite[Theorem 4.2]{Iom10}.
\end{proof}

\subsection{Identifying minimizers}\label{identifying}

\begin{proof}[Proof of Theorem B]
We first treat the harmonic and geometric means.
With a couple of lemmas below,
we show the uniqueness of minimizer 
for slightly more general $\psi$ than the harmonic/geometric mean cases.

\begin{lemma}\label{classify1}
Let $\psi$ be an arithmetic function such that
 $c_\psi<\psi(2)$. Let $\alpha\in[c_\psi,\psi(2)]$. Then $I_{\psi\circ b_1}(\alpha)=0$ if and only if $\alpha=\psi(2)$.
\end{lemma}
\begin{proof}
Let $\alpha\in[c_\psi,\psi(2))$. By the assumption on $\psi$,
there exists a constant $\chi_\alpha>0$ such that 
for any sequence $\{\mu_k\}_{k=1}^\infty$ in $\mathcal M_\phi(T)$ such that
$\displaystyle{\lim_{k\to\infty}}$$\int\psi\circ b_1d\mu_k=\alpha$,
$\displaystyle{\liminf_{k\to\infty}}\chi(\mu_k)\geq \chi_\alpha.$
In particular, each $\mu_k$ is an expanding measure.
From Lemma \ref{upbound}, for sufficiently large $k$,
\begin{align*}F(\mu_k)=\chi(\mu_k)\left(\dim(\mu_k)-1\right)
\leq \chi_\alpha\left(
\mathcal L(\chi_\alpha)-1\right)<0.\end{align*}
Hence, $I_{\psi\circ b_1}(\alpha)>0$ holds.
\end{proof}
Lemma \ref{classify1} with $\psi(n)=1/n$ finishes the proof of Theorem B(a).

\begin{lemma}\label{classify2}
Let $\psi$ be an arithmetic function such that $\psi(2)<d_\psi$ and
\begin{equation}\label{reika}\limsup_{n\to\infty}\frac{\psi(n)}{\log n}<\infty.\end{equation}
Let $\alpha\in[\psi(2),\infty)$. Then $I_{\psi\circ b_1}(\alpha)=0$  if and only if $\alpha=\psi(2)$.
\end{lemma}

\begin{proof}
Considering $\psi-\psi(2)$ instead of $\psi$, we may assume $\psi(2)=0$ with no loss of generality.
Let $\alpha\in(0,d_\psi)$. We show that there exists a constant $\chi_\alpha>0$ such that 
for any sequence $\{\mu_k\}_{k=1}^\infty$ in $\mathcal M_\phi(T)$ such that
$\displaystyle{\lim_{k\to\infty}}$$\int\psi\circ b_1d\mu_k=\alpha$,
\begin{equation}\label{rate-lem}
\displaystyle{\liminf_{k\to\infty}}\chi(\mu_k)\geq \chi_\alpha.
\end{equation}
This yields $I_{\psi\circ b_1}(\alpha)>0$ in the same way as the proof of Lemma \ref{classify1}.
 
We have
\begin{equation*}
\begin{split}\int \psi\circ b_1 d\mu_k&=\int_{[0,1/2)}\psi\circ b_1d\mu_k+\int_{[1/2,1)}\psi\circ b_1d\mu_k\\
&= \psi(2)\mu_k([0,1/2))+\int_{[1/2,1)}\psi\circ b_1d\mu_k\\
&= \int_{[1/2,1)}\psi\circ b_1d\mu_k\end{split}\end{equation*}
Let $\epsilon\in(0,\alpha)$.
Then for all sufficiently large $n$,
$$\int_{[1/2,1)}\psi\circ b_1d\mu_k>\alpha-\epsilon.$$
Since $T'(x)=-2\log(1-x)$, \eqref{reika} implies there exists a constant
$C>0$ such that $\log T'\geq C\psi\circ b_1$
on $[1/2,1)$.
Therefore
\begin{align*}
\chi(\mu_k)&=\int_{[0,1/2)}\log T'd\mu_k+\int_{[1/2,1)}\log T' d\mu_k\\
&\geq\int_{[1/2,1)}\log T' d\mu_k\\
&\geq C\int_{[1/2,1)}\psi\circ b_1d\mu_k\\
&>C(\alpha-\epsilon)>0.
\end{align*}
Put $\chi_\alpha=C(\alpha-\epsilon)$.
Then \eqref{rate-lem} holds.
\end{proof}

Lemma \ref{classify2} with $\psi(n)=\log n$ finishes the proof of Theorem B(b).

It is left to treat the arithmetic mean.
Let $J\subset[2,\infty)$ be a bounded interval.
For each integer $n\geq2/\lambda(J)$, fix an integer $z_n$ such that
$$\frac{z_n}{n}\in\left(\inf J,\inf J+\frac{\lambda(J)}{2}\right).$$
Put
$$B_n= T^{-n}\left[1-\frac{1}{z_n-2(n-1)-1},1-\frac{1}{z_n-2(n-1)}\right)\cap \bigcap_{j=0}^{n-1}T^{-j}
\left[0,\frac{1}{2}\right).$$
  If $x\in B_n$ then $$\frac{b_1(x)+\cdots b_n(x)}{n}\in\left(\inf J,\inf J+\frac{\lambda(J)}{2}\right)\subset J.$$
    We use the sequence $\{c_n\}_{n=0}^\infty$ used in the proof of Lemma \ref{distortion}.
  Let $\gamma>2$.
  For sufficiently large $n$, Lemma \ref{distortion} yields
  \begin{align*}
\frac{\lambda(B_n)}{\lambda([0,c_{n-1}])}&\geq  n^{-\gamma}\frac{\lambda(T^nB_n)}{\lambda(T^n[0,c_{n-1}])}\\
&\geq n^{-\gamma}\lambda\left(\left[1-\frac{1}{z_n-2(n-1)-1},1-\frac{1}{z_n-2(n-1)}\right)\right)\\
&\geq\frac{n^{-\gamma}}{(z_n-2(n-1))^2}\\
&\geq\frac{n^{-\gamma-2}}{(\sup J-2)^2}.
\end{align*}
The property $c_nn\to1$ as $n\to\infty$ (see \cite[Lemma 2, Corollary]{Tha83}) implies that
for any $\rho>0$ and sufficiently large $n$,
\begin{equation}
\begin{split}\label{a-mean}\lambda\left\{x\in[0,1)\colon\frac{b_1(x)+\cdots b_n(x)}{n}\in J\right\}
&\geq\lambda(B_n)\geq\frac{n^{-\gamma-\rho-3}}{(\sup J-2)^2}.
\end{split}\end{equation}
Hence, the sub-exponential lower large deviation bound in Theorem B(c) holds.

Now, let $\alpha\in[2,\infty)$.
For any $\epsilon>0$ we have
\begin{align*}
0&=\lim_{n\to\infty}\frac{1}{n}\log\lambda\left\{x\in[0,1)\colon\frac{b_1(x)+\cdots b_n(x)}{n}\in(\alpha-\epsilon,\alpha+\epsilon)\right\}\\
&= -\inf_{(\alpha-\epsilon,\alpha+\epsilon)} I_{ b_1}.\end{align*}
The first equality is from \eqref{a-mean} with $J=(\alpha-\epsilon,\alpha+\epsilon)$, and
the second one is from Theorem A.
The lower semi-continuity of the rate function yields $I_{b_1}(\alpha)=0$.
This completes the proof of Theorem B(c). \end{proof}

\subsection{Rate functions with unbounded flat part}\label{flat-part}

\begin{proof}[Proof of Theorem C]
Let $\psi$ be an arithmetic function such that
$\displaystyle{\liminf_{n\to\infty}}$$\psi(n)\geq0$ and 
$\displaystyle{\limsup_{n\to\infty}}$$\psi(n)=\infty.$
The `only if' part is a consequence of Lemma \ref{classify2}.
We show the `if' part. Assume that
$$\limsup_{n\to\infty}\frac{\psi(n)}{\log n}=\infty.$$ 
\begin{lemma}\label{seq}
There exists a sequence 
$\{\mu_k\}_{k=1}^\infty$ of expanding measures such that 
$$\lim_{k\to\infty}\chi(\mu_k)=0,$$
$$\int \psi\circ b_1d\mu_k<\infty\quad\text{for every $k\geq1$}\ \text{ and }\
\lim_{k\to\infty}\int \psi\circ b_1d\mu_k=\infty.$$ 
\end{lemma}
\noindent Since $h(\mu_k)\leq\chi(\mu_k)$ it follows that $\displaystyle{\lim_{k\to\infty}}$$F(\mu_k)=0$, and therefore $\displaystyle{\liminf_{\alpha\to\infty}}I_{\psi\circ b_1}(\alpha)=0$.
Since $I_{\psi\circ b_1}(\psi(2))=0$ and the rate function is convex, $I_{\psi\circ b_1}(\alpha)=0$
holds for all $\alpha\in[\psi(2),\infty)$.

We prove Lemma \ref{seq} modifying the proof of \cite[Lemma 4.4]{JaeTak}.
Take an increasing sequence $\{n_k\}_{k=1}^\infty$ of positive integers $\geq2$ 
and an increasing sequence
 $\{r_k\}_{k=1}^\infty$ of positive real numbers $\geq1$ such that 
 $\displaystyle{\lim_{k\to\infty}} r_{k}/\log n_k=\infty$ and
 $\displaystyle{\lim_{k\to\infty}}\psi(n_k)/r_{k}=\infty$
 (For example, given $n_{k-1}\geq2$ define $n_k$ to be the smallest $n\geq n_{k-1}$ such that $\psi(n)/\log n\geq k$,
 and put $r_{k}=\sqrt{k}\log n_k$).
By Lemma \ref{lalpha}, there exists an ergodic expanding measure with dimension arbitrarily close to $1$.
As in the proof of Lemma \ref{lowlem}, using
Birkhoff's ergodic theorem and Shannon-McMillan-Breiman's theorem one can approximate each of these measures
with a finite number of cylinders, and therefore
it is possible to take a sequence $\{\nu_k\}_{k=1}^\infty$ of expanding measures
  such that $\chi(\nu_k)\geq 1/\sqrt{r_k/\log n_k}$
and $\int \psi \circ b_1d\nu_k<\infty$ hold for every $k\geq1$
and $\displaystyle{\lim_{k\to\infty}}\dim(\nu_k)=1$.
Let $\delta_k$ denote the unit point mass at the fixed point of $T$ in $\left[1-1/(n_k-1),1-1/n_k\right)$,
and put
$$\mu_k=\left(1-\frac{1}{r_k}\right)\nu_k+\frac{1}{r_k}\delta_{k}.$$
Then $h(\delta_k)=0$,
$\chi(\delta_k)\leq 2\log n_k$ and $\int \psi\circ b_1d\mu_k<\infty$. Since $\psi(n_k)\geq0$,
$$\int \psi\circ b_1d\mu_k\geq \frac{1}{r_k}\int \psi\circ b_1d\delta_k=\frac{\psi(n_k)}{r_k},$$
which goes to $\infty$ as $k$ increases.
For any $c>1$ there exists $k_0\geq1$ such that for every $k\geq k_0$,
$\chi(\mu_k)\leq c\chi(\nu_k).$ Hence
$$\dim(\mu_k)=\frac{h(\mu_k)}{\chi(\mu_k)}\geq\frac{\left(1-1/r_k\right)h(\nu_k)}{c\chi(\nu_k)},$$
which implies
$\displaystyle{\liminf_{k\to\infty}}\dim(\mu_k)\geq 1/c$. Decreasing $c$ to $1$ yields
$\displaystyle{\lim_{k\to\infty}}\dim(\mu_k)=1$.

If $\displaystyle{\limsup_{k\to\infty}}\chi(\mu_k)>0$, then from Lemma \ref{lalpha} it follows that
$$\mathcal L\left(\displaystyle{\limsup_{k\to\infty}}\chi(\mu_k)\right)=1,$$
contradicting the strictly monotone decreasing property of the Lyapunov spectrum \cite[Theorem 4.2]{Iom10}.
This completes the proof of Lemma \ref{seq} and hence that of Theorem C.
\end{proof}


\subsection*{Acknowledgments} 
I thank Ryoki Fukushima, Hiroaki Ito, Johannes Jaerisch for fruitful discussions.
This research was partially supported by the JSPS KAKENHI 
16KT0021, 19K21835. 
Part of this paper was written during the conference
``Thermodynamic Formalism: Dynamical Systems, Statistical Properties and their Applications`` at CIRM Marseille December 2019.
I thank Mark Pollicott and Sandro Vaienti for their hospitality during the conference.

\end{document}